\documentclass{amsart2010}
\usepackage{graphicx}
\usepackage{amsmath}
\usepackage{amsfonts}
\usepackage{amssymb}

\newtheorem{thm}{Theorem}

\newtheorem{lem}[thm]{Lemma}

\theoremstyle{definition}

\newtheorem*{theorem*}{Theorem}
\newcommand\numberthis{\addtocounter{equation}{1}\tag{\theequation}}

\theoremstyle{remark}

\numberwithin{equation}{section}
\newcommand{\norm}[1]{\left\Vert#1\right\Vert}
\newcommand{\abs}[1]{\left\vert#1\right\vert}

\newcommand{\To}{\longrightarrow}

\def\<{\langle}
\def\>{\rangle}
\begin{document}
\title[]{A Hardy space analysis of the Bá\'{a}ez-Duarte criterion for the RH}
\author{S. Waleed Noor}%
\address{IMECC, Universidade Estadual de Campinas, Campinas-SP, Brazil.}
\email{$\mathrm{waleed@ime.unicamp.br}$}

\begin{abstract} In this article, methods from sub-Hardy Hilbert spaces such as the de Branges-Rovnyak spaces and local Dirichlet spaces are used to investigate B\'{a}ez-Duarte's Hilbert space reformulation of the Riemann hypothesis (RH).  
\end{abstract}

{\subjclass[2010]{Primary; Secondary}}
\maketitle{}

\section*{Introduction}
A classical reformulation of the Riemann hypothesis by Nyman and Beurling (see \cite{Beurling},\cite{Nyman}) says that all the non-trivial zeros of the $\zeta$-function lie on the critical line $\mathrm{Re}(s)=1/2$ if and only if the characteristic function $\chi_{(0,1)}$ belongs to the closed linear span in $L^2((0,1))$ of the set $\{f_\lambda:0\leq\lambda\leq1\}$, where $f_\lambda(x)=\{\lambda/x\}-\lambda\{1/x\}$ (here $\{x\}$ is the fractional part). Almost fifty years later a remarkable strengthening of this result by B\'{a}ez-Duarte \cite{Baez-Duarte} shows that we may replace $\lambda\in(0,1)$ by $\lambda=1/\ell$ for $\ell\geq 2$. There is an equivalent version of the B\'{a}ez-Duarte criterion in the weighted sequence space $\ell^2_\omega$ with inner product given by 

\begin{equation}\label{weighted l^2}
\langle x,y\rangle=\sum_{n=1}^\infty \frac{x(n)\overline{y(n)}}{n(n+1)}
\end{equation}
for sequences $x,y\in\mathcal{H}$ (see \cite[page 73]{Balazard Saias 4}). For each $k\geq 2$, let $r_k$ denote the sequence defined by
$r_k(n)=k\{n/k\}$. Then the B\'{a}ez-Duarte criterion  may be stated as follows:

\begin{thm}\label{Baez Duarte}
The RH is true if and only if $\boldsymbol{1}:=(1,1,1,\ldots)$ belongs to the closure of the linear span of $\{r_k:k\geq 2\}$ in $\ell^2_\omega$.
\end{thm}
 
 The plan of the paper is the following. Let $\mathcal{N}$ denote the linear span of the functions
 \[
 h_k(z)=\frac{1}{1-z}\log\left(\frac{1+z+\ldots+z^{k-1}}{k}\right)
 \]
 for $k\geq 2$, which all belong to the Hardy space $H^2$ (see Lemma \ref{H2 Membership lemma}). In Section 2 a unitary equivalent version of Theorem \ref{Baez Duarte} for the Hardy space $H^2$ is presented. In particular, the RH holds if and only if the constant $1$ belongs to the closure of $\mathcal{N}$ in $H^2$  (see Theorem \ref{Main Theorem}). Section 3 introduces a multiplicative semigroup of weighted composition operators $\{W_n:n\geq1\}$ on $H^2$ and shows that the constant 1 (appearing in Theorem \ref{Main Theorem}) may be replaced by any cyclic vector for $\{W_n:n\geq1\}$ in $H^2$. It follows that the RH is equivalent to the density of $\mathcal{N}$ in $H^2$ (see Theorem \ref{main thm extension}). Section 4 proves  that $(I-S)\mathcal{N}$ is dense in $H^2$, where $S$ is the shift operator on $H^2$ (see Theorem \ref{(I-S)hk}). This central result has the following remarkable consequence. That $\mathcal{N}$ is dense in $H^2$ with respect to the compact-open topology (see Theorem \ref{density comp-open}). Since convergence in $H^2$ implies convergence in the compact-open topology, this may be viewed as a \emph{weak} form of the RH. Section 5 shows that $\mathcal{N}^\perp$ is in a sense small by proving that 
 \[
 \mathcal{N}^\perp\cap\mathcal{D}_{\delta_1}=\{0\}
 \]
 where $\mathcal{D}_{\delta_1}$ is the local Dirichlet space at $1$ (which is dense in $H^2$), and in particular that $\mathcal{N}^\perp$ contains no function holomorphic on a neighborhood of the closed unit disk $\overline{\mathbb{D}}$ (see Theorem \ref{N perp}). Section 6 shows that the cyclic vectors for $\{W_n:n\geq1\}$ in $H^2$ are properly embedded into the set of all $2$-periodic functions $\phi$ on $(0,\infty)$ having the property that the span of its dilates
 $
\{\phi(nx):n\geq 1\}
 $
 is dense in $L^2(0,1)$ (see Theorem \ref{PDCP embedding}). The characterization of all such $\phi$ is a famous open problem known as the \emph{Periodic Dilation Completeness Problem}.

 \section{Background}
\subsection{The Hardy-Hilbert space} We denote by $\mathbb{D}$ and $\mathbb{T}$ the open unit disk and the unit circle respectively. A holomorphic function $f$ on $\mathbb{D}$ belongs to the Hardy-Hilbert space $H^2$ if
\[
||f||_{H^2}=\sup_{0\leq r<1}\left(\frac{1}{2\pi}\int_0^{2\pi}|f(re^{i\theta})|^2d\theta\right)^{1/2}<\infty.
\]
The space $H^2$ is a Hilbert space with inner product
\[
\langle f,g \rangle=\sum_{n=0}^\infty a_n\overline{b_n},
\]
where $(a_n)_{n\in\mathbb{N}}$ and $(b_n)_{n\in\mathbb{N}}$ are the  Maclaurin coefficients for $f$ and $g$ respectively. Similarly $H^\infty$ denotes the space of bounded holomorphic functions  defined on $\mathbb{D}$. For any $f\in H^2$ and $\zeta\in\mathbb{T}$, the radial limit $f^*(\zeta):=\lim_{r\to 1^-}f(r\zeta)$ exists $m$-a.e. on $\mathbb{T}$, where $m$ denotes the normalized Lebesgue measure on $\mathbb{T}$.

\subsection{A weighted Bergman space} Let $\mathcal{A}$ be the Hilbert space of analytic functions $f(z)=\sum_{n=0}^\infty a_nz^n$ and $g(z)=\sum_{n=0}^\infty b_nz^n$ defined on $\mathbb{D}$ for which the inner product is given by
\begin{equation}\label{norm A}
\langle f,g\rangle:=\sum_{n=0}^\infty \frac{a_n\overline{b_n}}{(n+1)(n+2)}.
\end{equation}
 There also exists an area integral form of the corresponding $\mathcal{A}$-norm given by
\begin{equation}\label{integral norm A}
||f||^2_{\mathcal{A}}=\int_{\mathbb{D}}\abs{f(z)}^2(1-\abs{z}^2)dA(z)
\end{equation}
where $dA$ is the normalized area measure on $\mathbb{D}$. Comparing \eqref{weighted l^2} with \eqref{norm A} shows that the map
\begin{equation}\label{H and A isomorphism}
\Psi:(x(1),x(2),\ldots)\longmapsto \sum_{n=0}^\infty x(n+1)z^n
\end{equation}
is a canonical isometric isomorphism of $\ell^2_\bold\omega$ onto $\mathcal{A}$.

The text \cite{Hedenmalm-Zhu} is a modern reference for such weighted Bergman spaces. Also if $\mathrm{Hol}(\mathbb{D})$ is the space of all holomorphic functions on $\mathbb{D}$ and $T:\mathrm{Hol}(\mathbb{D})\to \mathrm{Hol}(\mathbb{D})$ is an operator defined by
\begin{equation}\label{unitary map Hardy Bergman}
Tg(z):=\frac{((1-z)g(z))'}{1-z},
\end{equation}
then $T$ restricted to $H^2$ is an isometric isomorphism onto $\mathcal{A}$ (see Lemma 7.2.3 \cite{Dirichlet book}). Hence $\Phi:=T^{-1}\circ \Psi$ is an isometric isomorphism of $\ell^2_\bold\omega$ onto $H^2$. Therefore to obtain a reformulation of the B\'{a}ez-Duarte Theorem in $H^2$, we need to calculate $\Phi\boldsymbol{1}$ and $\Phi r_k$ for $k\geq 2$. But to do so we shall need some results about local Dirichlet spaces.

\subsection{Generalized Dirichlet spaces} Let $\mu$ be a finite positive Borel measure on $\mathbb{T}$, and let $P\mu$ denote its Poisson integral. The \emph{generalized Dirichlet space} $\mathcal{D}_\mu$ consists of $f\in H^2$ satisfying
\[
\mathcal{D}_\mu(f):=\int_\mathbb{D}\abs{f'(z)}^2P\mu(z)dA(z)<\infty.
\]
Then $\mathcal{D}_\mu$ is a Hilbert space with norm $\norm{f}_{\mathcal{D}_\mu}^2:=\norm{f}_2^2+\mathcal{D}_\mu(f)$. If $\mu=m$, then $\mathcal{D}_m$ is the classicial Dirichlet space. If $\mu=\delta_\zeta$ is the Dirac measure at $\zeta\in\mathbb{T}$, then $\mathcal{D}_{\delta_\zeta}$ is called the \emph{local Dirichlet space} at $\zeta$ and in particular
\begin{equation}\label{Dirichlet integral}
\mathcal{D}_{\delta_\zeta}(f)=\int_\mathbb{D}\abs{f'(z)}^2\frac{1-\abs{z^2}}{\abs{z-\zeta}^2}dA(z).
\end{equation}
 The recent book \cite{Dirichlet book} contains a comprehensive treatment of local Dirichlet spaces and the following result establishes a criterion for their membership.
\begin{thm}\label{Local Dirichlet membership}(See  \cite[Thm. 7.2.1]{Dirichlet book})
	Let $\zeta\in\mathbb{T}$ and $f\in\mathrm{Hol}(\mathbb{D})$. Then
	$\mathcal{D}_{\delta_\zeta}(f)<\infty$ if and only if \[f(z)=a+(z-\zeta)g(z)\] for some $g\in H^2$ and $a\in \mathbb{C}$. In this case $\mathcal{D}_{\delta_\zeta}(f)=\norm{g}_2^2$ and 
	\[
	a=f^*(\zeta):=\lim_{r\to1^{-}}f(r\zeta).
	\]
\end{thm}
Each local Dirichlet space $\mathcal{D}_{\delta_\zeta}$ is a proper subspace of $H^2$ and it has the distinctive property that evaluation at the boundary $f\mapsto f^*(\zeta) $ is a bounded linear functional \cite[Thm. 8.1.2 (ii)]{Dirichlet book}.

\subsection{The de Branges-Rovnyak spaces}
Given $\psi\in L^\infty(\mathbb{T})$, the corresponding Toeplitz operator $T_\psi:H^2\to H^2$ is defined by
\[
T_\psi f:=P_+(\psi f )
\] 
where $P_+ : L^2(\mathbb{T})\to H^2$ denotes the orthogonal projection of $L^2(\mathbb{T})$ onto $H^2$. 
Clearly $T_\psi$ is a bounded operator on $H^2$ with $||T_\psi|| \leq||\psi||_{L^\infty}$. If $h\in H^\infty$, then
$T_h$ is simply the operator of multiplication by $h$ and its adjoint is $T_{\overline{h}}$. Given $b$ in the closed unit ball of $H^\infty$, the \emph{de Branges-Rovnyak}
space $\mathcal{H}(b)$ is the image of $H^2$ under the operator $(I -T_bT_{\overline{b}})^{1/2}$. A norm is defined
on $\mathcal{H}(b)$ making $(I -T_bT_{\overline{b}})^{1/2}$ a partial isometry from $H^2$ onto $\mathcal{H}(b)$. If $b\equiv0$ then $\mathcal{H}(b)=H^2$, and if $b$ is inner then $\mathcal{H}(b)=(bH^2)^\perp$ is the model subspace of $H^2$.
The recent two-volume work (\cite{Hb book vol 1}\cite{Hb book vol 2}) is an encyclopedic reference for these spaces.

The general theory of $\mathcal{H}(b)$ spaces divides into two distinct cases, according to whether b is
an extreme point or a non-extreme point of the unit ball of $H^\infty$. 

We shall only be concerned with the non-extreme case which is best illustrated by the next result (see \cite[Chapter 6]{Hb book vol 1} and \cite[Sects. IV-6 and V-1]{Sarason book}).

\begin{thm}\label{nonextreme H(b)}
	Let $b\in H^\infty$ with $||b||_{H^\infty}\leq1$. The following are equivalent:
	\begin{enumerate}
		\item $b$ is a non-extreme point of the unit ball of $H^\infty$, 
		\item $\log(1-|b^*|^2)\in L^1(\mathbb{T})$, 
		\item $\mathcal{H}(b)$ contains all functions holomorphic in a neighborhood of $\overline{\mathbb{D}}$.
	\end{enumerate}
\end{thm}

When $b$ is non-extreme there exists a unique outer function $a\in H^\infty$ such that $a(0)>0$ and $|a^*|^2+|b^*|^2=1$ a.e. on $\mathbb{T}$. In this situation $(b,a)$ is usually called a \emph{pair} and the function $b/a$ belongs to the Smirnov class $N^+$ of quotients $p/q$ where $p,q\in H^\infty$ and $q$ is an outer function. That all $N^+$ functions arise as the quotient of a pair associated to a non-extreme function was shown by Sarason (see \cite{unbounded toplitz}). 

In \cite{unbounded toplitz}, Sarason also demonstrated how $\mathcal{H}(b)$ spaces appear naturally as the domains of some unbounded Toeplitz operators. Let $\varphi$ be holomorphic in $\mathbb{D}$ and $T_\varphi$ the operator of multiplication by $\varphi$ on the domain 
\begin{equation}\label{Toeplitz domain}
\mathrm{dom}(T_\varphi)=\{f\in H^2:\varphi f\in H^2\}.
\end{equation}
Then $T_\varphi$ is a closed operator, and $\mathrm{dom}(T_\varphi)$ is dense in $H^2$ if and only if $\varphi\in N^+$ (see \cite[Lemma 5.2]{unbounded toplitz}). In this case its adjoint $T_\varphi^*$ is also densely defined and closed. In fact the domain of $T_\varphi^*$ is a de Branges-Rovnyak space. 

\begin{thm}\label{dom of adjoint}(See \cite[Prop. 5.4]{unbounded toplitz})
	Let $\varphi$ be a nonzero function in $N^+$ with $\varphi=b/a$, where $(b,a)$ is the associated pair. Then $\mathrm{dom}(T_\varphi^*)=\mathcal{H}(b)$.
\end{thm}

If $\varphi$ is a rational function in $N^+$ the corresponding pair $(b,a)$ is also rational (see \cite[Remark. 3.2]{unbounded toplitz}). Recently Constara and Ransford \cite{Constara-Ransford} characterized the rational pairs $(b,a)$ for which $\mathcal{H}(b)$ is a generalized Dirichlet space.

\begin{thm}\label{which DR=Dirichlet}(See \cite[Theorem 4.1]{Constara-Ransford}) Let $(b,a)$ be a rational pair and $\mu$ a finite positive measure on $\mathbb{T}$. Then $\mathcal{H}(b)=\mathcal{D}_\mu$ if and only if
	\begin{enumerate}
		\item the zeros of $a$ on $\mathbb{T}$ are all simple, and
		\item the support of $\mu$ is exactly equal to this set of zeros.
	\end{enumerate}
	
\end{thm}

These ideas will be used in Section 5 to investigate the orthogonal complement of the functions $\{h_k:k\geq 2\}$ in $H^2$.

\section{The B\'{a}ez-Duarte criterion in $H^2$}
 The first main objective is to obtain a unitary equivalent version of B\'{a}ez-Duarte's theorem (Theorem \ref{Baez Duarte}) in $H^2$ upon which to base the rest of our analysis.
 
 \begin{thm}\label{Main Theorem} For each $k\geq 2$, define
 	\[
 	h_k(z)=\frac{1}{1-z}\log\left(\frac{1+z+\ldots+z^{k-1}}{k}\right).
 	\]  Then the Riemann hypothesis holds if and only if the constant $1$ belongs to the closed linear span of $\{h_k:k\geq 2\}$ in $H^2$.	
 \end{thm}
 
 In order to prove this, we must show that $-1=\Phi\boldsymbol{1}$ and $h_k=\Phi r_k$ for $k\geq 2$, where $\Phi:=T^{-1}\circ \Psi:\ell^2_\omega\to H^2$ is an isometric isomorphism (see subsection 1.2). 
 
 We first find $R:=\Psi\boldsymbol{1}$ and $R_k:=\Psi r_k$, which belong to the weighted Bergman space $\mathcal{A}$. Then 
\[
R(z)=\frac{1}{1-z}  ,\quad  R_k(z)=\frac{1}{1-z}[\log(1+z+\ldots+z^{k-1})]'
\]
for each $k=2,3,\ldots$ (note that $R_1\equiv 0$). The expression for $R$ is trivial. For $R_k$ we first note that the sequence
 $r_k(n)=k\{n/k\}$ is periodic with $k$ distinct integer terms $\{1,2,\ldots,k-1,0,\ldots\}$. Hence collecting terms with common coefficients gives
\begin{align*}
R_k(z)&=\sum_{n=0}^\infty z^{nk}+2\sum_{n=0}^\infty z^{nk+1}+\ldots+ (k-1)\sum_{n=0}^\infty z^{nk+k-2}=\sum_{m=1}^{k-1}m\sum_{n=0}^\infty z^{nk+m-1} \\
&=\sum_{m=1}^{k-1}m\frac{z^{m-1}}{1-z^k}=\frac{1}{1-z^k}\sum_{m=1}^{k-1}mz^{m-1}=\frac{1}{1-z}\left[\frac{(1+z+\ldots+z^{k-1})'}{1+z+\ldots+z^{k-1}}\right] \\
&=\frac{1}{1-z}[\log(1+z+\ldots+z^{k-1})]'. \numberthis \label{R_k}
\end{align*}

Next we calculate $T^{-1}R$ and $T^{-1}R_k$ in $H^2$. It is easy to see that $T(-1)= R$ and hence $-1=\Phi\boldsymbol{1}$. But finding the $T^{-1}R_k$ is not as straightforward because $T$ is not injective on $\mathrm{Hol}(\mathbb{D})$. 
                       
\begin{lem}\label{H2 Membership lemma} For each non-zero $c$ and integer $k\geq 2$, define the function
	\[
	h_{k,c}(z)=\frac{1}{1-z}\log\left(\frac{1+z+\ldots+z^{k-1}}{c}\right).
	\]  
	Then $Th_{k,c}=R_k$ for each $c$, but $h_{k,c}\in H^2$ if and only if $c=k$. 
\end{lem}
\begin{proof}
Let $s_k(z):=\log(1+z+\ldots+z^{k-1})$ for $k\geq 2$. Since $R_k$ belongs to $\mathcal{A}$ for $k\geq 2$, by \eqref{integral norm A}, \eqref{Dirichlet integral} and \eqref{R_k} we have
\begin{align*}
\mathcal{D}_{\delta_1}(s_k)&=\int_{\mathbb{D}}\abs{[\log(1+z+\ldots+z^{k-1})]'}^2\frac{1-\abs{z}^2}{\abs{z-1}^2}dA(z) \\
&=\int_{\mathbb{D}}\abs{R_k(z)}^2(1-\abs{z}^2)dA(z)\\
&=\norm{R_k}_\mathcal{A}^2<\infty.
\end{align*}	
Therefore $s_k$ belongs to the local Dirichlet space $\mathcal{D}_{\delta_1}$. By Theorem \ref{Local Dirichlet membership} there exists $f_k\in H^2$ such that $s_k(z)=s_k^*(1)+(z-1)f_k(z)=\log k+(z-1)f_k(z)$ and it follows immediately that
\[
f_k(z)=\frac{1}{z-1}\log\left(\frac{1+z+\ldots+z^{k-1}}{k}\right).
\]
Hence $h_{k,k}=-f_k\in H^2$. Since clearly $Th_{k,c}=R_k$ for each non-zero $c$ and $T$ is injective on $H^2$, therefore $c=k$ is the only value for which $h_{k,c}\in H^2$.
\end{proof}

Therefore with $h_k:=h_{k,k}$ for all $k\geq 2$ this concludes the proof of Theorem \ref{Main Theorem}. We end this section by giving an alternate proof of the fact that $h_k\in H^2$ for $k\geq 2$ which also provides an explicit formula for the Maclaurin coefficients of the $h_k$.\footnote[1]{The author wishes to thank the anonymous referee for this alternate proof.} 

For $k,n\in\mathbb{N}$, define the function $[k|n]$ to be $1$ if $k$ divides $n$ and $0$ otherwise, and note that $\log(1-z^k)=-\sum_{j\geq 1}z^{jk}/j=-\sum_{n\geq 1}k[k|n]z^n/n$. Then for $k\geq 2$
\begin{align*}
h_k(z)&=\frac{1}{1-z}(\log(1-z^k)-\log(1-z)-\log k) \\
&=\frac{1}{1-z}(-\sum_{n\geq 1}k[k|n]\frac{z^n}{n}+\sum_{n\geq 1}\frac{z^n}{n}-\log k) \\
&=-(\log k)\sum_{n\geq 0}z^n+\frac{1}{1-z}\sum_{n\geq 1}\frac{z^n}{n}(1-k[k|n])=\sum_{n\geq 0}c_n(k)z^n
\end{align*}
where $c_n(k)=-\log k +\sum_{j=1}^n\frac{1}{j}(1-k[k|j])=H(n)-H(n/k)-\log k$ and the function $H(x):=\sum_{n\leq x}\frac{1}{n}$ for $x>0$ and $H(0)=0$. The Euler-Maclaurin summation formula gives
\[
H(x)=\log x+\lambda-\frac{\{x\}}{x}+\int_x^\infty\{t\}\frac{dt}{t^2}=\log x+\lambda+O\left(\frac{1}{x}\right)
\]
where $\lambda$ is Euler's constant and hence $c_n(k)=O(\frac{k}{n})$. Therefore $h_k\in H^2$. \qed

\section{A weighted composition semigroup}
In \cite{Bagchi} Bagchi showed that in addition to Theorem \ref{Baez Duarte}, the RH is equivalent to the density of $\mathrm{span}\{r_k:k\geq 2\}$ in $\ell^2_\omega$. A key ingredient in his proof is a multiplicative semigroup of operators which leave $\mathrm{span}\{r_k:k\geq 2\}$ invariant (see \cite[Theorem 7]{Bagchi}). The relation of invariant subspaces of semigroups with the RH has been evident since the thesis of Nyman \cite{Nyman} (see also \cite{Baez-Duarte 1} and \cite{Nikolski}). 

 For each $n\geq 1$, let $W_n$ be a weighted composition operator on $H^2$ defined by
\begin{equation}\label{W_n}
W_nf(z)=(1+z+\ldots+z^{n-1})f(z^n)=\frac{1-z^n}{1-z}f(z^n).
\end{equation}
Note that each $W_n$ is bounded on $H^2$, $W_1=I$ and $W_mW_n=W_{mn}$ for each $m,n\geq 1$. Hence $\{W_n:n\geq1\}$ is a multiplicative semigroup on $H^2$. Now if we write 
\begin{equation}\label{Alternative h_k}
h_k(z)=\frac{1}{1-z}\left(\log(1-z^k)-\log(1-z)-\log k\right)
\end{equation}
then it is easy to see that $W_nh_k=h_{nk}-h_n$ for all $k,n\geq 1$ (where $h_1\equiv 0$). Hence the linear span of $\{h_k:k\geq 2\}$ is invariant under
$\{W_n:n\geq1\}$. A vector $f\in H^2$ is called a \emph{cyclic} vector for an operator semigroup $\{S_n:n\geq 1\}$ if $\mathrm{span}\{S_nf:n\geq 1\}$ is dense in $H^2$. Hence the following combines Bagchi's result and a generalization of Theorem \ref{Main Theorem}.

\begin{thm}\label{main thm extension} The following statements are equivalent
	\begin{enumerate}
		\item The Riemann hypothesis,
		\item the closure of span$\{h_k:k\geq 2\}$ contains a cyclic vector for $\{W_n:n\geq1\}$, 
		\item span$\{h_k:k\geq 2\}$ is dense in $H^2$.
	\end{enumerate}	
\end{thm}
\begin{proof} The equivalence $(1)\leftrightarrow (3)$ is just Bagchi's result transfered to $H^2$ via the isomorphism $\Phi:\ell^2_\omega\to H^2$. The implication $(1)\rightarrow (2)$ follows from Theorem \ref{Main Theorem} and the fact that $1$ is a cyclic vector for the semigroup $\{W_n:n\geq1\}$. Indeed $(W_n1)(z)=1+z+\ldots+z^{n-1}$ for all $n\geq 1$ so $\mathrm{span}\{W_n1:n\geq 1\}$ contains all analytic polynomials and is hence dense in $H^2$. Finally $(2) \rightarrow (3)$ because if the closure of span$\{h_k:k\geq 2\}$ contains a cyclic vector $f\in H^2$, then it also contains the dense manifold $\mathrm{span}\{W_n f:n\geq 1\}$ by the invariance of span$\{h_k:k\geq 2\}$ under $\{W_n:n\geq1\}$. 
	\end{proof}

In Section 6, we shall see that characterizing the cyclic vectors for $\{W_n:n\geq1\}$ is intimately related to another famous open problem known as the \emph{Periodic Dilation Completeness Problem} (see \cite{Hedenmalm-Seip} and \cite{Nikolski}). 

\section{The density of $\mathrm{span}\{(I-S)h_k:k\geq 2\}$ in $H^2$}
Let $S=T_z$ be the shift operator on $H^2$. Since $I-S$ has dense range (because $I-S^*$ is injective), therefore $\mathrm{span}\{(I-S)h_k:k\geq 2\}$ is dense in $H^2$ under the RH by Theorem \ref{main thm extension}. Proving that this statement is unconditionally true is the main objective of this section and it will play a central role in the rest of this work.

\begin{thm}\label{(I-S)hk}
	The span of $\{(I-S)h_k:k\geq 2\}$ is dense in $H^2$.
\end{thm}

Since convergence in $H^2$ implies uniform convergence on compact subsets of $\mathbb{D}$, we obtain a \emph{weak} version of the RH. 

\begin{thm}\label{density comp-open}
The span of $\{h_k:k\geq 2\}$ is dense in $H^2$ with the compact-open topology.
\end{thm}
\begin{proof} The formal inverse of $I-S$ is the Toeplitz operator $T_\varphi$ of multiplication by the function $\varphi(z)=\frac{1}{1-z}$. Although $T_\varphi$ is unbounded on $H^2$ (otherwise Theorem \ref{(I-S)hk} would imply the RH), it is still continuous on $H^2$ with the compact-open topology. Therefore the result follows immediately from Theorem \ref{(I-S)hk}.
\end{proof}

Define the multiplicative operator semigroup $\{T_n:n\geq 1\}$ on $H^2$ by
\begin{equation}\label{T_n}
T_nf(z)=f(z^n).
\end{equation}
Then by \eqref{W_n} and \eqref{T_n} it is easily seen that
\begin{equation}\label{quasiconjugacy}
T_n(I-S)=(I-S)W_n \ \ \ \ \forall \ n\geq 1.
\end{equation}
Recall that span$\{h_k:k\geq 2\}$ is invariant under $\{W_n:n\geq1\}$ (see Section 3), and hence \eqref{quasiconjugacy} implies that span$\{(I-S)h_k:k\geq 2\}$ is invariant under $\{T_n:n\geq 1\}$. So to prove Theorem \ref{(I-S)hk}, it is enough to prove that the closure of span$\{(I-S)h_k:k\geq 2\}$ contains a cyclic vector for $\{T_n:n\geq 1\}$. And the cyclic vector we consider is $1-z$. Indeed, if $f\in H^2$ is orthogonal to each $T_n(1-z)=1-z^n$ then $\widehat{f}(0)=\widehat{f}(n)$ for all $n\geq 1$ and hence $f\equiv 0$. Hence the next result completes the proof of Theorem \ref{(I-S)hk}.
 
\begin{lem}\label{Necessary condition 2} The series $\sum_{k=2}^\infty\frac{\mu(k)}{k}(I-S)h_k$ converges to $1-z$ in $H^2$, where $\mu$ is the M\"{o}bius function.
\end{lem} 
Recall that the M\"{o}bius function is defined on $\mathbb{N}$ by $\mu(k)=(-1)^s$ if $k$ is the product of $s$ distinct primes, and $\mu(k)=0$ otherwise. In the proof we shall need the \emph{Prime Number Theorem} in the equivalent forms 

\begin{equation}\label{PNT mu}
\sum_{k=1}^\infty \frac{\mu(k)}{k}=0 \ \ \mathrm{and} \ \ \sum_{k=1}^\infty\frac{\mu(k)\log k}{k}=-1
\end{equation}
(see \cite[Thm. 4.16]{Apostol} and \cite[p. 185, Excercise 16]{Montgomery-Vaughan}).
\begin{proof}
	One has to prove that 
	\begin{equation}\label{App. in H^2}
	\norm{\sum_{k=2}^n\frac{\mu(k)}{k}(I-S)h_k+z-1}_{H^2}\To 0 
	\end{equation}
	as $n\to\infty$. Since $(I-S)h_k(z)=\log(1-z^k)-\log(1-z)-\log k$ (see \eqref{Alternative h_k}), we get
	
	\begin{equation}\label{eq1 thm 5}
	\sum_{k=2}^n\frac{\mu(k)}{k}(I-S)h_k(z)=
	\sum_{k=1}^n\frac{\mu(k)}{k}\log(1-z^k)-\sum_{k=1}^n\frac{\mu(k)}{k}[\log(1-z)+\log{k}].  
	\end{equation}
	First note that the last sum on the right of \eqref{eq1 thm 5} tends to $1$ as $n\to\infty$ by \eqref{PNT mu}. Writing the first sum as a double sum after noting that $\log(1-z^k)=-\sum_{j=1}^\infty \frac{z^{jk}}{j}$, interchanging the order of summation and using the basic identity $\sum_{d|j}\mu(d)=\left[\frac{1}{j}\right]$ if $j\geq 1$ \cite[Thm 2.1]{Apostol} ($[x]$ denotes the integer part of $x$), we get
	\begin{align*}
&\sum_{k=1}^n\frac{\mu(k)}{k}\log(1-z^k)=-	\sum_{k=1}^n\frac{\mu(k)}{k}\sum_{j=1}^\infty \frac{z^{jk}}{j}=-	\sum_{k=1}^n\mu(k)\sum_{j=1}^\infty \frac{z^{jk}}{jk} \\
	&=-\sum_{j=1}^\infty  \frac{z^{j}}{j}\sum_{\substack{d|j \\ 1\leq d\leq n}}\mu(d)=-\sum_{j=1}^n \frac{z^{j}}{j}\sum_{d|j}\mu(d)-\sum_{j=n+1}^\infty \frac{z^{j}}{j}\sum_{\substack{d|j \\ 1\leq d\leq n}}\mu(d) \\
	&=-\sum_{j=1}^n \frac{z^{j}}{j}\left[\frac{1}{j}\right]-\sum_{j=n+1}^\infty \frac{z^{j}}{j}\sum_{\substack{d|j \\ 1\leq d\leq n}}\mu(d)=-z-\phi_n(z). \numberthis \label{eq 2 thm 5}
	\end{align*}	
	Therefore by \eqref{eq1 thm 5} and \eqref{eq 2 thm 5}, we will  prove \eqref{App. in H^2} once we prove that $\norm{\phi_n}_{H^2}\to 0$ as $n\to\infty$. Since
	\begin{equation}\label{phi_n}
	\phi_n(z)=\sum_{j=n+1}^\infty \frac{z^{j}}{j}\sum_{\substack{d|j \\ 1\leq d\leq n}}\mu(d)
	\end{equation}
	and if $\sigma(n)$ denotes the number of divisors of $n$, then it follows that
	\[
	|\sum_{\substack{d|j \\ 1\leq d\leq n}}\mu(d)|\leq \sum_{d|j}1=  \sigma(j).
	\]
	The function $\sigma$ satisfies the relation $\sigma(n)=o(n^{\epsilon})$ for every $\epsilon>0$ \cite[p. 296]{Apostol}. In particular, $\sigma(n)\lesssim n^\epsilon$ for some $0<\epsilon<\frac{1}{2}$, and therefore by \eqref{phi_n}
	\[
	||\phi_n||_{H^2}^2\leq\sum_{j=n+1}^\infty\frac{\sigma(j)^2}{j^2}\lesssim\sum_{j=n+1}^\infty j^{2\epsilon-2}\To 0
	\]
	as $n\to\infty$. This proves \eqref{App. in H^2} and hence the lemma.
\end{proof}

\section{Functions orthogonal to $\{h_k:k\geq 2\}$ }
The RH is equivalent to $\{h_k:k\geq 2\}^\perp$ being trivial $\{0\}$ (see Theorem \ref{main thm extension}). The main result of this section shows that $\{h_k:k\geq 2\}^\perp$ is indeed in a sense very small.

\begin{thm}\label{N perp}
	We have
	\[
	\{h_k:k\geq 2\}^\perp\cap\mathcal{D}_{\delta_1}=\{0\}
	\]
	where $\mathcal{D}_{\delta_1}$ is the local Dirichlet space at $1$. In particular $\{h_k:k\geq 2\}^\perp$ contains no function holomorphic on a neighborhood of the closed unit disk $\overline{\mathbb{D}}$.
\end{thm}

The key idea is to use the formal inverse $T_\varphi$ of $I-S$, where $\varphi(z)=\frac{1}{1-z}$ is clearly an $N^+$ function. Then there is a pair $(b,a)$ associated with $\varphi$ where
\begin{equation}\label{a(z) golden}
a(z)=\frac{\gamma(1-z)}{(\gamma+1)-z}
\end{equation}
and $\gamma=\frac{1+\sqrt{5}}{2}$ is the \emph{golden ratio} (see \cite[page 284]{unbounded toplitz}). Therefore by Theorem \ref{dom of adjoint}, Theorem \ref{which DR=Dirichlet} and \eqref{a(z) golden} we immediately see that

\begin{equation}\label{domT* main}
\mathrm{dom}(T_\varphi^*)=\mathcal{H}(b)=\mathcal{D}_{\delta_1}
\end{equation} 
where $T_\varphi^*$ is the adjoint of $T_\varphi$ (see subsection 1.4).

\begin{proof} Let $g_k:=(I-S)h_k$ for each $k\geq 2$ and note that $\mathrm{span}\{g_k:k\geq 2\}$ is dense in $H^2$ by Theorem \ref{(I-S)hk}. Also note $g_k\in\mathrm{dom}(T_\varphi)$ because $h_k=T_\varphi g_k$ and by  \eqref{Toeplitz domain}. Now let $p$ be an element in $\{h_k:k\geq 2\}^\perp\cap\mathrm{dom}(T_\varphi^*)$. Hence for each $k\geq 2$, we have
	\[
	\langle T_\varphi^* p,g_k\rangle=\langle  p,T_\varphi g_k\rangle=\langle p,h_k\rangle=0.
	\]
	Therefore $T_\varphi^* p\equiv0$. But this implies that $p\equiv 0$, because 
		\[
		\langle p,T_\varphi f\rangle=\langle T_\varphi^*p, f\rangle=0
		\]
	for each $f\in\mathrm{dom}(T_\varphi)$ and the range of $T_\varphi$ is all of $H^2$ (it is the domain of $I-S$).
	Hence $\{h_k:k\geq 2\}^\perp\cap\mathrm{dom}(T_\varphi^*)=\{0\}$, \eqref{domT* main} and Theorem \ref{nonextreme H(b)} complete the proof.
	\end{proof}

\section{The Periodic Dilation Completeness Problem PDCP} The PDCP asks which $2$-periodic functions $\phi$ on $(0,\infty)$ have the property that 
\[
\mathrm{span}\{\phi(nx):n\geq 1\}
\]
is dense in $L^2(0,1)$. In this case we shall just say that $\phi$ is a \emph{PDCP function}. This difficult open problem was first considered independently by Wintner \cite{Wintner} and Beurling \cite{Beurling 2}. See \cite{Hedenmalm-Seip} and \cite{Nikolski} for beautiful modern treatments. The main result of this section shows that the cyclic vectors for $\{W_n:n\geq 1\}$ in $H^2$  (see Theorem \ref{main thm extension}) are properly embedded into the PDCP functions.

\begin{thm}\label{PDCP embedding}
	There exists an injective linear map $V:H^2\to L^2(0,1)$ such that if $f$ is a cyclic vector for $\{W_n:n\geq 1\}$ in $H^2$, then $Vf$ is a PDCP function.
\end{thm}
The function $Vf\in L^2(0,1)$ is defined on the whole real line by extending it as an odd $2$-periodic function.

\begin{proof}
Recall that the semigroups $\{W_n:n\geq 1\}$ and $\{T_n:n\geq 1\}$ satisfy the relation
\[
T_n(I-S)=(I-S)W_n \ \ \ \ \forall \ n\geq 1 .
\]
where $I-S$ has dense range in $H^2$ (see \eqref{quasiconjugacy}). It follows that if $\mathrm{span}\{W_n f:n\geq 1\}$ is dense in $H^2$ for some $f\in H^2$, then $\mathrm{span}\{T_n (I-S)f:n\geq 1\}$ must also be dense. 

So $f\mapsto (I-S)f$ maps cyclic vectors for $\{W_n:n\geq 1\}$ to cyclic vectors for $\{T_n:n\geq 1\}$. Let \[
H^2_0:=\{f\in H^2:f(0)=0\}=H^2\ominus\mathbb{C}
\]
 and note that $H^2_0$ is a reducing subspace for $T_n$ since  $T_n\mathbb{C}\subset\mathbb{C}$ and $T_n H^2_0\subset H^2_0$. Denote by $P$ the orthogonal projection of $H^2$ onto $H^2_0$. It follows that if $f$ is a cyclic vector for $\{T_n:n\geq 1\}$ in $H^2$ then $Pf$ is a cyclic vector for $\{T_n:n\geq 1\}$ restricted to $H^2_0$. Therefore 
\[
P(I-S):H^2\to H^2_0
\]
maps cyclic vectors for $\{W_n:n\geq 1\}$ into cyclic vectors for $\{T_n:n\geq 1\}$ restricted to $H^2_0$. Finally there is a unitary operator $U:H^2_0\to L^2(0,1)$ such that $f$ is cyclic for $\{T_n:n\geq 1\}$ in $H^2_0$ if and only if $Uf$ is a PDCP function  (see \cite[page 1707]{Nikolski}). In fact, it is defined by
\begin{equation}\label{U}
U: z^k\longmapsto e_k(x):=\sqrt{2}\sin(\pi kx)
\end{equation}
for each $k\geq 1$, where $(e_k)_{k\geq 1}$ is an orthonormal basis for $L^2(0,1)$. Therefore the operator 
\begin{equation}\label{P(I-S)}
V:=UP(I-S):H^2\to L^2(0,1)
\end{equation}
maps cyclic vectors for $\{W_n:n\geq 1\}$ into PDCP functions. It is injective since $\mathrm{Ker}(P)=\mathbb{C}$ and the inverse image of $\mathbb{C}$ under $I-S$ is $\{0\}$.
\end{proof}

Finally, we show that \emph{not all} PDCP functions belong to the range of $V$ \eqref{P(I-S)}. Wintner \cite{Wintner} showed that for $\mathrm{Re}(s)>1/2$ the function
\[
f_s(x)=\sum_{k\geq 1}k^{-s}\sqrt{2}\sin(\pi kx)
\]
is a PDCP function. We give an independent proof that $f_1$ is a PDCP function and that it does not belong to the range of $V$.

\begin{thm}  $f_1$ is a PDCP function that does not belong to the range of $V$.
\end{thm}
\begin{proof} Let $L(z):=\log(1-z)=-\sum_{k\geq 1}z^k/k$ and note that $U(-L)=f_1$ (see \eqref{U}). Hence it is enough to prove that $L$ is a cyclic vector for $\{T_n:n\geq 1\}$ in $H^2_0$. Note that since $(I-S)h_k(z)=\log(1-z^k)-\log(1-z)-\log k$ we have
	\[
	P(I-S)h_k=T_kL-T_1L
	\]
	and hence 
	\begin{equation}\label{L}
	P(\mathrm{span}\{(I-S)h_k:k\geq 2\})\subset\mathrm{span}\{T_nL:n\geq 1\}.
	\end{equation}
	By Theorem \ref{(I-S)hk} the left side of \eqref{L} is dense in $H^2_0$ and hence $L$ is cyclic. Therefore $f_1$ is a PDCP function. To prove that $f_1$ is not in the range of $V$, we show that $L$ is not in the range $P(I-S)$. The functions mapped onto $L$ by $P$ are of the form $\alpha+L$ for some $\alpha\in\mathbb{C}$. But $\alpha+L$ does not belong to $(I-S)H^2$ because $L^*(1)$ does not exist and $f^*(1)=0$ for all $f\in(I-S)H^2$ (see Theorem \ref{Local Dirichlet membership}).
		\end{proof}
	
\section*{Acknowledgement}
	This work has been partially supported by a FAPESP grant (17/09333-3).
\bibliographystyle{amsplain}

\begin{thebibliography}{00}
\bibitem{Apostol} T. M. Apostol, Introduction to analytic number theory. UTM Springer, 1976.


\bibitem{Bagchi} B. Bagchi, On Nyman, Beurling and B\'{a}ez-Duarte's Hilbert space reformulation of the Riemann hypothesis. Proc. Ind. Acad. Sci (Math. Sci.), 116(2), 137-146, 2006. 


\bibitem{Beurling} A. Beurling, A closure problem related to the Riemann zeta-function. Proc. Nat. Acad. Sci., 41, 312-314, 1955.

\bibitem{Beurling 2} A. Beurling, On the completeness of $\psi(nt)$ on $L^2(0,1)$, in Harmonic Analysis, Contemp. Mathematicians, The collected works of Arne Beurling, vol. 2, Birkhauser, Boston, 1989, p. 378-380.

\bibitem{Baez-Duarte} L. B\'{a}ez-Duarte, A strengthening of the Nyman-Beurling criterion for the Riemann hypothesis. Atti Acad. Naz. Lincei 14, 5-11, 2003.

\bibitem{Baez-Duarte 1} L. B\'{a}ez-Duarte, A Class of Invariant Unitary Operators. Adv. Math. 144 (1999) 1-12.

\bibitem{Balazard Saias 4} M. Balazard, E. Saias, Notes sur la fonction $\zeta$ de Riemann 4, Adv. Math. 188 (2004) 69-86.

\bibitem{Constara-Ransford} C. Constara, T. Ransford, Which de Branges-Rovnyak spaces are Dirichlet spaces (and vice versa)?. J. Funct. Anal. 265(12), 3204-3218 (2010)

\bibitem{Hb book vol 1} E. Fricain, J. Mashreghi, The theory of $\mathcal{H}_{b}$ spaces. Vol. 1, volume 20 of New Mathematical Monographs. Cambridge University Press, Cambridge (2016)

\bibitem{Hb book vol 2} E. Fricain, J. Mashreghi, The theory of $\mathcal{H}_{b}$ spaces. Vol. 2, volume 21 of New Mathematical Monographs. Cambridge University Press, Cambridge (2016)

\bibitem{Hedenmalm-Seip}  H. Hedenmalm, P. Lindqvist, and K. Seip. A Hilbert space of Dirichlet series and systems of
dilated functions in $L^2(0, 1)$. Duke Math. J., 86:1–37, 1997. MR 99i:42033

\bibitem{Hedenmalm-Zhu} H. Hedenmalm, B. Korenblum, K. Zhu, Theory of Bergman spaces. GTM Springer, volume 199, 2000.

\bibitem{Dirichlet book}  J. Mashreghi, K. Kellay, Omar El-Fallah, and T. Ransford, A primer on the Dirichlet space. Cambridge Tracts in Mathematics (203), Cambridge University Press, 2014.

\bibitem{Montgomery-Vaughan} H. L. Montgomery, R. C. Vaughan, Multiplicative number theory: 1. Classical theory. Cambridge Studies in Advanced Mathematics (97), Cambridge University Press, 2006.

\bibitem{Nikolski} N. Nikolski, In a shadow of the RH: cyclic vectors of the Hardy spaces on the Hilbert multidisc. Ann. Inst. Fourier, 62(5), 1601-1626 (2012).

\bibitem{Nyman} B. Nyman, On some groups and semigroups of translations. Thesis, Uppsala, 1950.


\bibitem{unbounded toplitz} D. Sarason, Unbounded Toeplitz operators. Integr. Equ. Oper. Theory. 61(2), 281-298 (2008).

\bibitem{Sarason book} D. Sarason, Sub-Hardy Hilbert Spaces in the Unit Disk, John Wiley \& Sons Inc., New York, 1994.

\bibitem{Wintner} A. Wintner, Diophantine approximation and Hilbert's space. Amer. J. Math. 66 (1944), p.564-578.

\end{thebibliography}

\end{document}